\documentclass[preprint,12pt]{elsarticle} 
\usepackage[T1]{fontenc}
\usepackage{subfigure}
\usepackage{amssymb}
\usepackage{amsthm}

\journal{Int J Approx Reas}
\usepackage[english]{babel}
\usepackage{color,amsthm,hyperref}
\def\True{\mbox{True}}
\def\False{\mbox{False}}
\def\Void{\mbox{Void}}
\def\myscale{0.52} 
\def\C{{\cal C}}

\def\F{{\cal F}}

\def\I{{\cal I}}
\def\H{{\cal H}}
\def\K{{\cal K}}

\def\P{{\cal P}}
\def\S{{\cal S}}

\theoremstyle{definition} 
\newtheorem{definition}{Definition}
\newtheorem{proposition}{Proposition}
\newtheorem{theorem}{Theorem}
\newtheorem{corollary}{Corollary}
\newtheorem{remark}{Remark}
\newtheorem{algorithm}{Algorithm}
\newtheorem{example}{Example}
\begin{document}
\begin{frontmatter}
\title{Probabilistic entailment in the setting of coherence: The role of quasi conjunction and inclusion relation \tnoteref{t1}}
\tnotetext[t1]{This paper is a revised and expanded version of \cite{GiSa11b}.}

\author[ag]{Angelo Gilio}
\ead{angelo.gilio@sbai.uniroma1.it}
\author[gs]{Giuseppe  Sanfilippo\corref{cor2}}
\ead{giuseppe.sanfilippo@unipa.it}
\cortext[cor2]{Principal corresponding author}

\address[ag]{Dipartimento di Scienze di Base e Applicate per l'Ingegneria,
University of Rome ``La Sapienza'',
Via Antonio Scarpa 16, 00161 Roma, Italy}
\address[gs]{Dipartimento di Scienze Statistiche e Matematiche ``S. Vianelli'',University of Palermo,
Viale delle Scienze ed.13, 90128 Palermo, Italy
}

\begin{abstract}
In this paper, by adopting a coherence-based probabilistic approach to default reasoning, we focus the study on the logical operation of quasi conjunction and the Goodman-Nguyen inclusion relation for conditional events. We recall that quasi conjunction is a basic notion for defining consistency of conditional knowledge bases. By deepening some results given in a previous paper we show that, given any finite family of conditional events $\mathcal{F}$  and any nonempty subset $\mathcal{S}$ of $\mathcal{F}$, the family  $\mathcal{F}$ p-entails the quasi conjunction $\mathcal{C}(\mathcal{S})$; then, given any conditional event $E|H$, we analyze the equivalence between p-entailment of $E|H$ from $\mathcal{F}$ and p-entailment of $E|H$ from $\mathcal{C}(\mathcal{S})$, where $\mathcal{S}$ is some nonempty subset of $\mathcal{F}$. We also illustrate some alternative theorems related with p-consistency and p-entailment. Finally, we deepen the study of the connections between the notions of p-entailment and inclusion relation by introducing for a pair $(\mathcal{F},E|H)$ the (possibly empty) class $\K$ of the subsets $\mathcal{S}$ of $\mathcal{F}$ such that $\mathcal{C}(\mathcal{S})$ implies $E|H$. We show that the class $\K$ satisfies many properties; in particular $\K$ is additive and has a greatest element which can be determined by applying a suitable algorithm.
\end{abstract}

\end{frontmatter}

\section{Introduction}
Probabilistic reasoning is a basic tool for the treatment of uncertainty in many applications of statistics and artificial intelligence; in particular, it is useful for a flexible numerical approach to inference rules in nonmonotonic reasoning, for the psychology of uncertain reasoning and for the management of uncertainty on semantic web (see, e.g.,  \cite{FPMK11,GiOv12,GiLu02,LuPS11,LuSt08,PfKl06,PfKl09}. \\
This work concerns nonmonotonic reasoning, an important topic in the field of artificial intelligence which has been studied by many authors, by using symbolic and/or numerical formalisms (see, e.g. \cite{BeDP97,2002BGLS-JANCL,BGLS05,CoSc02,DuboisPrade1994,Gili12,KrLM90}). As is well known, differently from classical (monotonic) logic, in (nonmonotonic) commonsense reasoning if a conclusion $C$ follows from some premises, then $C$ may be retracted when the set of premises is enlarged; that is, adding premises may invalidate previous conclusions.
Among the numerical formalisms connected with nonmonotonic reasoning, a remarkable theory is represented by the Adams probabilistic logic of conditionals (\cite{Adam75}), which can be developed with full generality in the setting of coherence. As is well known, based on the coherence principle of de Finetti (\cite{deFi70}), conditional probabilities can be directly assigned to conditional assertions, without assuming that conditioning events have a positive probability (see, e.g., \cite{BiGS03b,BiGS09,BiGS12,CaVa02,CoSc02,Gilio2002,Sanf12}).
We also recall that this approach does not require the assertion of complete distributions and is largely applied in statistical analysis and decision theory (see, for instance, \cite{BCTV02,BrCV12,CaGV03,CaLS07,CaRV10,CoVa06,LaSA12}).
 In Adams' work a basic notion is the quasi conjunction of conditionals, which has a strict relationship with the property of consistency of conditional knowledge bases. Quasi conjunction also plays a relevant role in the work of Dubois and Prade on conditional objects (\cite{DuboisPrade1994}, see also \cite{BeDP97}), where a suitable QAND rule has been introduced to characterize entailment from a
conditional knowledge base. Recently (\cite{GiSa10}, see also \cite{GiSa11}), we have studied some probabilistic aspects related with the QAND rule and with the conditional probabilistic logic of Adams.
We continue such a study in this paper by giving further results on the role played  by quasi conjunction and Goodman-Nguyen's inclusion relation in the probabilistic entailment under coherence. \\
The paper is organized as follows: In Section 2 we first recall some notions and results on coherence; then, we recall basic notions in probabilistic default reasoning; we recall the operation of quasi conjunction and the inclusion relation for conditional events; finally, we recall the notion of entailment for conditional objects. In Section 3 we give a result which shows the p-entailment from a family of conditional events $\mathcal{F}$ to the quasi conjunction $\mathcal{C}(\mathcal{S})$, for every nonempty subset $\mathcal{S}$ of $\mathcal{F}$; we give another result which analyzes many aspects connected with the equivalence between the p-entailment of a conditional event $E|H$ from $\mathcal{F}$ and the p-entailment of $E|H$ from $\mathcal{C}(\mathcal{S})$, where $\mathcal{S}$ is some nonempty subset of $\mathcal{F}$; then, we give some alternative theorems related with p-consistency and p-entailment.
In Section 4 we introduce for a pair $(\mathcal{F},E|H)$ the class $\mathcal{K}$ of the subsets $\mathcal{S}$ of $\mathcal{F}$ such that $\mathcal{C}(\mathcal{S})$ implies $E|H$. We show that $\mathcal{K}$ satisfies many properties and we give some examples; in particular, we show that  $\mathcal{K}$ is additive and has a greatest element (if any) which can be determined by a suitable  algorithm. In Section 5 we give some conclusions.
\section{Some Preliminary Notions}\label{PRELIM}
In this section we recall some basic notions and results on the following topics: (i) coherence of conditional probability assessments; (ii) probabilistic default reasoning; (iii) inclusion relation of Goodman-Nguyen and quasi conjunction of conditional events; (iv) entailment among conditional objects and QAND rule.
\subsection{Basic notions  on coherence} Given any event $E$ we use the same symbol to denote its indicator and we denote by $E^c$ the negation of $E$. Given any events $A$ and $B$, we simply write $A \subseteq B$ to denote that $A$ logically implies $B$. Moreover, we denote by $AB$ (resp., $A \vee B$) the logical intersection, or conjunction (resp., logical union, or disjunction) of $A$ and $B$.
We recall that $n$ events are said logically independent when there are no logical dependencies among them; this amounts to say that the number of atoms, or constituents, generated by them is $2^n$.
The conditional event $B|A$, with $A \neq \emptyset$, is looked at as a three-valued logical entity which is true, or false, or void, according to whether $AB$ is true, or $AB^c$ is true, or $A^c$ is true. Given a real function $P : \; \mathcal{F} \, \rightarrow \, \mathbb{R}$, where $\mathcal{F}$ is an arbitrary family of conditional events, let
us consider a subfamily $\mathcal{F}_n = \{E_1|H_1, \ldots, E_n|H_n\} \subseteq \mathcal{F}$,
and the vector $\mathcal{P}_n =(p_1, \ldots, p_n)$, where $p_i =
P(E_i|H_i) \, ,\;\; i = 1, \ldots, n$.
We denote by $\mathcal{H}_n$ the disjunction $H_1 \vee \cdots \vee H_n$.
Notice that, $E_iH_i \vee E_i^cH_i \vee H_i^c = \Omega$, $i = 1, \ldots,
n$, where $\Omega$ is the sure event; then, by expanding the expression
$\bigwedge_{i=1}^n(E_iH_i \vee E_i^cH_i \vee H_i^c)$,
we can represent $\Omega$ as the disjunction of $3^n$ logical
conjunctions, some of which may be impossible.  The remaining ones
are the constituents generated by the family $\mathcal{F}_n$. We denote by
$C_1, \ldots, C_m$ the constituents contained in $\mathcal{H}_n$ and (if
$\mathcal{H}_n \neq \Omega$) by $C_0$ the further constituent
$\mathcal{H}_n^c =
H_1^c \cdots H_n^c $, so that
\[
\mathcal{H}_n = C_1 \vee \cdots \vee C_m \,,\;\;\; \Omega =
\mathcal{H}_n^c \vee
\mathcal{H}_n = C_0 \vee C_1 \vee \cdots \vee C_m \,,\;\;\; m+1 \leq 3^n
\,.
\]
With the pair $(\mathcal{F}_n, \mathcal{P}_n$) we associate the random gain
${\mathcal{G}} = \sum_{i=1}^n s_iH_i(E_i - p_i)$,
where $s_1, \ldots, s_n$ are $n$ arbitrary real numbers. Let $g_h$
be the value of $\mathcal{G}$ when $C_h$ is true; of course $g_0 = 0$.
Denoting by ${\mathcal{G}}|\mathcal{H}_n$ the restriction of ${\mathcal{G}}$ to
$\mathcal{H}_n$, it is ${\mathcal{G}}|\mathcal{H}_n \in \{g_1, \ldots, g_m\}$.
Then, we have
\begin{definition} {\rm The function $P$ defined on $\mathcal{F}$ is {\em coherent} if and only if, for every integer $n$, for every finite sub-family $\mathcal{F}_n$
$\subseteq \mathcal{F}$ and for every $s_1, \ldots, s_n$, one has:
$\min \; {\mathcal{G}}|\mathcal{H}_n \; \leq 0 \leq \max \;
{\mathcal{G}}|\mathcal{H}_n$. }\end{definition}
From the previous definition it immediately follows that in order $P$ be coherent it must be $P(E|H) \in [0,1]$ for every $E|H \in \mathcal{F}$. If $P$ is coherent it is called a {\em conditional probability on $\mathcal{F}$} (see, e.g., \cite{CoSc02}). Given any family $\mathcal{F}^*$, with $\mathcal{F} \subset \mathcal{F}^*$, and any function $P^*$ defined on $\mathcal{F}^*$, assuming $P$ coherent, we say that $P^*$ is a coherent extension of $P$ if the following conditions are satisfied: (i) $P^*$ is coherent; (ii) the restriction of $P^*$ to $\F$ coincides with $P$, that is for every $E|H \in \mathcal{F}$ it holds that $P^*(E|H) = P(E|H)$. In particular, if $\F^*$ contains the set of unconditional events $\{EH, H : E|H \in \mathcal{F}\}$, then  for every $E|H \in \mathcal{F}$ the coherent extension $P^*$ satisfies  the compound probability theorem $P^*(EH)=P^*(H)P(E|H)$ and hence, when $P^*(H)>0$, we can represent $P(E|H)$ as the ratio $\frac{P^*(EH)}{P^*(H)}$.  \\
With each $C_h$ contained in  $\mathcal{H}_n$ we associate a point
$Q_h = (q_{h1}, \ldots, q_{hn})$, where $q_{hj} = 1$, or 0, or $p_j$, according to whether $C_h \subseteq E_jH_j$, or $C_h \subseteq E_j^cH_j$, or $C_h \subseteq H_j^c$.
Denoting by $\mathcal{I}$ the convex hull of  $Q_1, \ldots, Q_m$,
based on the penalty criterion, the result below can be proved (\cite{Gilio90,Gilio92}, see also \cite{GiSa11a}).

\begin{theorem}\label{CNES}{\rm
The function $P$ is coherent if and only if, for every finite subfamily $ \mathcal{F}_n \subseteq \mathcal{F}$, one has  $\mathcal{P}_n \in \mathcal{I}$. } \end{theorem}
The condition $\mathcal{P}_n \in \mathcal{I}$ amounts to solvability of the
following system $\Sigma$ in the unknowns $\lambda_1, \ldots,
\lambda_m$
\[
(\Sigma) \hspace{1 cm}
\left\{
\begin{array}{l}
\sum_{h=1}^m q_{hj} \lambda_h = p_j \; , \; \; j = 1, \ldots, n \,
; \\[0.5ex]
\sum_{h=1}^m \lambda_h = 1 \; , \; \; \lambda_h \geq 0 \, ,
\; h = 1, \ldots, m.
\end{array}
\right.
\]
{\em Checking coherence of the assessment $\mathcal{P}_n$ on $\mathcal{F}_n$.} \\
Let $S$ be the set of solutions $\Lambda = (\lambda_1, \ldots,
\lambda_m)$ of the system $\Sigma$. Then, define
\begin{equation}\label{PHI-I0}\begin{array}{l}
\Phi_j(\Lambda) = \Phi_j(\lambda_1, \ldots, \lambda_m) = \sum_{r :
C_r \subseteq H_j} \lambda_r \; , \; \; \; j = 1, \ldots, n \,;
\\
M_j  =  \max_{\Lambda \in S } \; \Phi_j(\Lambda) \; , \; \; \; j = 1, \ldots, n \,;
\;\;\; I_0  =  \{ j \, : \, M_j=0 \} \,.
\end{array}\end{equation}
Notice that $I_0 \subset \{1, \ldots,
n\}$, where $\subset$ means strict inclusion. We denote by $(\mathcal{F}_0, \mathcal{P}_0)$ the pair associated with $I_0$, that is $\mathcal{F}_0=\{E_j|H_j\in \mathcal{F}_n: j\in I_0\}$ and $\mathcal{P}_0=(p_j: j\in I_0)$.
Given the pair $(\mathcal{F}_n,\mathcal{P}_n)$ and a subset $J \subset \{1,
\ldots, n\}$, we denote by $(\mathcal{F}_J, \mathcal{P}_J)$ the pair associated with
$J$ and by $\Sigma_J$ the corresponding system.
We observe that $\Sigma_J$ is solvable if and only if $\mathcal{P}_J  \in \mathcal{I}_J$,
where $\mathcal{I}_J$ is the convex hull associated with the pair $( \mathcal{F}_J,
\mathcal{P}_J)$. Then, we have (\cite{Gili93,Gili95}, see also \cite{BiGS03})
\begin{theorem}\label{GILIO-93}{\rm
Given a probability assessment $\mathcal{P}_n$ on the family $\mathcal{F}_n$, if
the system $\Sigma$ associated with $(\mathcal{F}_n,\mathcal{P}_n)$ is solvable, then for every $J\subset \{1,\ldots,n\}$, such that $J\setminus I_0\neq \emptyset$, the system $\Sigma_J$ associated with $(\mathcal{F}_J,\mathcal{P}_J)$ is solvable too.}
\end{theorem}
By the previous results, we obtain
\begin{theorem}{\rm The assessment $\mathcal{P}_n$ on $\mathcal{F}_n$ is coherent if and only if the following conditions are satisfied: (i)
$\mathcal{P}_n \in \mathcal{I}$; (ii) if $I_0 \neq \emptyset$, then $\mathcal{P}_0$ is coherent.
}\end{theorem}
Then, we can check coherence by the following procedure:
\begin{algorithm}\label{Alg1}
{\rm Let  the pair $(\mathcal{F}_n,\mathcal{P}_n)$ be given.
\begin{enumerate}
\item Construct the system $\Sigma$  and check its solvability.
\item If the system $\Sigma$ is not solvable then $\mathcal{P}_n$ is
not coherent and the procedure stops, otherwise compute the set
$I_0$.
\item If $I_0 = \emptyset$ then $\mathcal{P}_n$ is coherent and the procedure
stops; otherwise set $(\mathcal{F}_n, \mathcal{P}_n) = (\mathcal{F}_0, \mathcal{P}_0)$ and repeat
steps 1-3.
\end{enumerate}
}\end{algorithm}
We remark  that, in the algorithm, $\Sigma$ is initially the system associated with $(\mathcal{F}_n,\mathcal{P}_n)$; after the first cycle $\Sigma$ is the system associated with $(\mathcal{F}_0,\mathcal{P}_0)$, and so on. If, after $k+1$ cycles, the algorithm stops at step 2 because  $\Sigma$ is unsolvable, then denoting by $(\mathcal{F}_k,\mathcal{P}_k)$ the pair associated with $\Sigma$, we have that $\mathcal{P}_k$ is not coherent and, of course, $\mathcal{P}_n$ is not coherent too.
\subsection{Basic notions on probabilistic default reasoning} We now give in the setting of coherence the notions of p-consistency and p-entailment of Adams (\cite{Adam75}). Given a family of $n$ conditional events $\mathcal{F}_n = \{E_i|H_i \, , \; i=1,\ldots,n\}$, we define below the notions of p-consistency and p-entailment for $\mathcal{F}_n$.
\begin{definition}\label{PC}{\rm
The family of conditional events $\mathcal{F}_n = \{E_i|H_i \, , \; i=1,\ldots,n\}$ is \linebreak {\em p-consistent} if and only if, for every set of lower bounds $\{\alpha_i,
i=1,\ldots,n\}$, with $\alpha_i \in [0,1)$, there exists a coherent probability
assessment $\{p_i, i=1,\ldots,n\}$ on $\mathcal{F}_n$, with $p_i =
P(E_i|H_i)$, such that $p_i \geq \alpha_i, i=1,\ldots,n$. }
\end{definition}

\begin{remark} \label{REM-PCONSIST}
Notice that p-consistency of $\mathcal{F}_n$ can be introduced by an equivalent condition, as shown by the result below (\cite[Thm 4.5]{2002BGLS-JANCL},\cite[Thm 8]{Gilio2002}).
\end{remark}
\begin{theorem}\label{THM-S-PCONSIST}
A family of conditional events $\F_n$ is p-consistent if and only if the assessment $(p_1,p_2,\ldots,p_n)=(1,1,\ldots,1)$ on $\mathcal{F}_n$ is coherent.
\end{theorem}

\begin{definition}\label{PE}{\rm
A p-consistent family  $\mathcal{F}_n = \{E_i|H_i \, , \; i=1,\ldots,n\}$ {\em p-entails}  $B|A$, denoted $\mathcal{F}_n \; \Rightarrow_p \; B|A$, if and only if  there exists a nonempty subset, of  $\mathcal{F}_n$, $\mathcal{S} = \{E_i|H_i,\; i \in J\}$ with  $J \subseteq \{1,\ldots,n\}$,  such that, for every $\alpha \in [0,1)$, there
exists a set  $\{\alpha_i, i \in J\}$, with $\alpha_i \in [0,1)$,
such that for all coherent assessments $(z, p_i, i \in J)$  on $\{B|A, E_i|H_i \,, \; i \in J\}$, with $z = P(B|A)$ and $p_i = P(E_i|H_i)$, if  $p_i
\geq \alpha_i$ for every  $i \in J$, then $z \geq \alpha$. }
\end{definition}
As we show in Theorem \ref{ENTAIL-CS}, a p-consistent family $\mathcal{F}_n$ p-entails $B|A$ if and only if, given any coherent assessment $(p_1,\ldots,p_n,z)$ on $\mathcal{F}_n \cup \{B|A\}$, from the condition $p_1=\cdots=p_n=1$  it follows $z=1$ (see also \cite[Thm 4.9]{2002BGLS-JANCL}). Of course, when $\mathcal{F}_n$ p-entails $\{B|A\}$, there may be coherent assessments $(p_1,\ldots,p_n,z)$ with $z \neq 1$, but in such case $p_i \neq 1$ for at least one index $i$.

We give below the notion of p-entailment between two families of conditional events $\Gamma$ and $\mathcal{F}$.
\begin{definition}\label{ENTAIL-FAM}{\rm
Given two p-consistent finite families of conditional events $\mathcal{F}$ and $\Gamma$, we say that $\mathcal{F}$ p-entails $\Gamma$ if $\mathcal{F}$ p-entails $E|H$, for every $E|H \in \Gamma$.
}\end{definition}
\noindent {\em Transitive property:} Of course, p-entailment is transitive; that is, given three p-consistent families of conditional events $\mathcal{F}, \Gamma, \mathcal{U}$, if $\mathcal{F} \; \Rightarrow_p \; \Gamma$ and $\Gamma \; \Rightarrow_p \; \mathcal{U}$, then $\mathcal{F} \; \Rightarrow_p \; \mathcal{U}$.
\begin{remark}\label{ENTAIL-SUB} Notice that, from Definition \ref{PE}, we trivially have that $\mathcal{F}$ p-entails $E|H$, for every $E|H \in \mathcal{F}$; then, by Definition \ref{ENTAIL-FAM}, it immediately follows
\begin{equation}
\mathcal{F} \; \Rightarrow_p \; \mathcal{S} \;,\;\; \forall \, \mathcal{S} \subseteq \mathcal{F} \,,\; \mathcal{S} \neq \emptyset \,.
\end{equation}
\end{remark}
\subsection{Quasi conjunction and inclusion relation} We recall below the notion of quasi conjunction of conditional events.
\begin{definition}{\rm
Given any events $A, H$, $B, K$, with $H \neq
\emptyset, K \neq \emptyset$, the quasi conjunction of the
conditional events $A|H$ and $B|K$, as defined in \cite{Adam75}, is
the conditional event $\mathcal{C}(A|H,B|K) = (AH \vee H^c) \wedge (BK \vee K^c)|(H \vee K)$, or equivalently $\mathcal{C}(A|H,B|K) = (AHBK \vee AHK^c \vee H^cBK)|(H \vee K)$.
More in general, given a family of $n$ conditional events $\mathcal{F}_n=\{E_i|H_i,\, i=1,\ldots,n\}$, the quasi conjunction of the conditional events in $\mathcal{F}_n$ is the conditional event  \[
\mathcal{C}(\mathcal{F}_n)=\mathcal{C}(E_1|H_1,\ldots,E_n|H_n) = \bigwedge_{i=1}^n (E_iH_i\vee H_i^c)\big |(\bigvee_{i=1}^n H_i) \,.
\]
}\end{definition}
The operation of quasi conjunction is associative; that is, for every subset $J \subset \{1,\ldots,n\}$, defining $\Gamma = \{1,\ldots,n\} \setminus J$ and
\[
\mathcal{F}_J = \{E_j|H_j \in \F_n : j \in J\} \,,\;\; \mathcal{F}_{\Gamma} = \{E_r|H_r \in \F_n : r \in \Gamma\} \,,
\]
it holds that $\mathcal{C}(\mathcal{F}_n) = \mathcal{C}(\mathcal{F}_J \cup \mathcal{F}_{\Gamma}) = \mathcal{C}[\mathcal{C}(\mathcal{F}_J), \mathcal{C}(\mathcal{F}_{\Gamma})]$. \\
If $A, H, B, K$ are logically independent, then we have (\cite{Gilio2004}): \\ (i) the probability assessment $(x,y)$ on the family $\{A|H, B|K\}$ is coherent for every $(x,y) \in [0,1]^2$; \\ (ii) given a coherent
assessment $(x,y)$ on $\{A|H, B|K\}$, the probability assessment
$(x,y,z)$ on $\{A|H, B|K, \mathcal{C}(A|H,B|K)\}$, where $z = P[\mathcal{C}(A|H,B|K)]$, is a coherent extension of $(x,y)$ if and only if: $l \leq z \leq u$, where
\[
l = \mbox{max}(x+y-1,0) \,;\;\;\;\; u = \left\{\begin{array}{ll} \frac{x+y-2xy}{1-xy}, & max \, \{x,y\} < 1, \\ 1, & max \, \{x,y\} = 1.
\end{array} \right.
\]
We observe that $l=T_L(x,y)$ and $u=S_0^H(x,y)$, where $T_L$ is the Lukasiewicz t-norm and $S_0^H$ is the Hamacher t-conorm with parameter $\lambda=0$.
A general probabilistic analysis for quasi conjunction has been given in \cite{GiSa10,GiSa11}. \\ Logical operations of 'conjunction', 'disjunction' and 'iterated conditioning' for conditional events have been introduced in \cite{GiSa12b}. \\
The notion of logical inclusion among events has been generalized to conditional events by Goodman and Nguyen in \cite{GoNg88}. We recall below this notion.
\begin{definition}{\rm Given two conditional events $A|H$ and $B|K$, we say that $A|H$ implies $B|K$, denoted by $A|H \subseteq B|K$, if and only if $AH$ {\em true} implies $BK$ {\em true} and $B^cK$ {\em true} implies $A^cH$ {\em true}; that is
\[
A|H \subseteq B|K \; \Longleftrightarrow \; AH \subseteq BK \; \mbox{and} \; B^cK \subseteq A^cH \,.
\]
}\end{definition}
Given any conditional events $A|H, B|K$, we have
\[
A|H = B|K  \; \Longleftrightarrow \; A|H \subseteq B|K \; \mbox{and}  \; B|K \subseteq A|H \,;
\]
that is: $A|H = B|K  \; \Longleftrightarrow \; AH = BK \; \mbox{and}  \; H = K$. Moreover
\begin{equation}\label{NOT-GN}
\begin{array}{l}
A|H \subseteq B|K \Longleftrightarrow AHB^cK=H^cB^cK=AHK^c=\emptyset\,, \\
 \vspace{-0.2cm} \\ A|H \nsubseteq B|K \; \Longleftrightarrow \; AHB^cK \vee H^cB^cK \vee AHK^c \neq \emptyset\,.
\end{array}
\end{equation}
The inclusion relation among conditional events, with an extension to conditional random quantities, has been recently studied in \cite{PeVi12}.
\subsection{Conditional Objects}
Based on a three-valued calculus of conditional objects (which are a qualitative counterpart of conditional probabilities) a logic for nonmonotonic reasoning has been proposed by Dubois and Prade in \cite{DuboisPrade1994} (see also \cite{BeDP97}). Such a  three-valued semantics of conditional objects was first proposed for conditional events in \cite{deFi35}. The conditional object associated with a pair $(p,q)$ of Boolean propositions is denoted by $q|p$, which reads "$q$ given $p$", and concerning logical operations we can look at conditional objects as conditional events because the three-valued semantics is the same. In particular, the quasi conjunction of conditional objects exactly corresponds to quasi conjunction of conditional events and the logical entailment $\models$ among conditional objects corresponds to the inclusion relation $\subseteq$ among conditional events. In the nonmonotonic logic developed by Dubois and Prade the quasi conjunction plays a key role, as is shown by the following definition of entailment of a conditional object from a finite conditional knowledge base $K = \{q_i|p_i, i=1,\ldots,n\}$ \linebreak (see \cite{DuboisPrade1994}, Def. 1).
\begin{definition}\label{ENT-DP}{\rm
$K$ entails a conditional object $q|p$, denoted $K \models q|p$, if and only if either there exists a nonempty subset $\mathcal{S}$ of $K$ such that for the quasi conjunction $\mathcal{C}(\mathcal{S})$ it holds that  $\mathcal{C}(\mathcal{S}) \models q|p$, or $p \models q$.
}\end{definition}
Then, assuming $K$ finite, the following inference rule (\cite{Adam75,BeDP97,DuboisPrade1994}) follows:
\[
\begin{array}{llll}
\mbox{(QAND)} &  \hspace{3cm} & K  \;\models  \; \mathcal{C}(K) \;. &\hspace{4cm}
\end{array}
\]

\section{Probabilistic Entailment and Quasi Conjunction}

The next result, related to Adams' work, generalizes Theorem 6 given in \cite{GiSa10} and deepens the probabilistic semantics of the QAND rule in the framework of coherence.
\begin{theorem}\label{ENT-QC}{\rm
Given a p-consistent family of conditional events $\mathcal{F}_n$, for every nonempty subfamily $\mathcal{S}=\{E_i|H_i, i=1,\ldots,s\} \subseteq \mathcal{F}_n$, we have
\begin{equation}\label{ENTAIL-QCSUB}
(a) \;\; QAND: \;\;\; \mathcal{S} \;\Rightarrow_p \; \mathcal{C}(\mathcal{S}) \,;\;\;\;\;\;\;
(b) \;\; \mathcal{F}_n \;\Rightarrow_p \; \mathcal{C}(\mathcal{S}) \,.
\end{equation}
}\end{theorem}
\begin{proof} (a) We preliminarily observe that the p-consistency of $\mathcal{F}_n$ implies the p-consistency of $\mathcal{S}$. In order to prove that $\mathcal{S}$ p-entails $\mathcal{C}(\mathcal{S})$ we have to show that  for every $\varepsilon \in (0,1]$ there exist $\delta_1 \in (0,1],\, \ldots, \delta_s \in (0,1]$ such that, for every coherent assessment $(p_1,\ldots,p_s, z)$ on $\mathcal{S} \cup \{\mathcal{C}(\mathcal{S})\}$, where $p_i = P(E_i|H_i)$, $z=P(\mathcal{C}(\mathcal{S}))$, if $p_1 \geq 1 - \delta_1, \ldots, p_s \geq 1 - \delta_s$,  then $z \geq 1 - \varepsilon$. \\ We distinguish two cases: (i) the events $E_i,H_i, i=1,\ldots,s$, are logically independent; (ii) the events $E_i,H_i, i=1,\ldots,s$, are logically dependent. \\
(i) We recall that the case $s=2$, with $\mathcal{S}=\{E_1|H_1, E_2|H_2\}$, has been already examined in \cite{Gilio2004}, by observing that, given any
coherent assessment $(p_1,p_2,z)$ on the family $\{E_1|H_1, E_2|H_2, \mathcal{C}(E_1|H_1, E_2|H_2)\}$ and any number $\gamma \in [0,1)$, for every
$\alpha_1 \in [\gamma,1)$, $\alpha_2 \in [\gamma,1)$, with $ \alpha_1 +
\alpha_2 \geq \gamma + 1 $, one has
\begin{equation}\label{PRINT}
(p_1,p_2) \in [\alpha_1,1] \times [\alpha_2,1] \;\; \Longrightarrow
\;\; z \; \geq \alpha_1 + \alpha_2 - 1 \; \geq \; \gamma \,.
\end{equation}
We observe that, for $\gamma = 1 - \varepsilon,\, \alpha_1 = 1 - \delta_1,\, \alpha_2 = 1 - \delta_2$, with $\alpha_1 + \alpha_2 \geq \gamma + 1$, i.e. $\delta_1+\delta_2 \leq \varepsilon$, formula (\ref{PRINT}) becomes
\[
p_1 \geq 1 - \delta_1,\, p_2 \geq 1 - \delta_2 \;\; \Longrightarrow
\;\; z \; \geq 1 - \delta_1 + 1 - \delta_2 - 1 \geq 1 - \varepsilon \,.
\]
More in general, denoting by $\mathcal{L}_{\gamma}$ the set of the coherent
assessments $(p_1,\ldots,p_s)$ on $\mathcal{S}$ such that, for each $(p_1,\ldots,p_s) \in \mathcal{L}_{\gamma}$, one has $P[\mathcal{C}(\mathcal{S})] \geq \gamma$, it can be proved (see \cite{GiSa10}, Theorem 4; see also \cite{GiSa11}, Theorem 9) that
\[
\mathcal{L}_{\gamma} = \{(p_1, \ldots, p_s) \in [0,1]^s : p_1 + \cdots +
p_s  \geq \gamma + s - 1\} \,.
\]
In particular, given any $\varepsilon > 0$, it is
\[
\mathcal{L}_{1-\varepsilon} = \{(p_1, \ldots, p_s) \in [0,1]^s : p_1 + \cdots +
p_s  \geq s - \varepsilon\} \,.
\]
Then, given any positive vector $(\delta_1, \ldots, \delta_s)$ in the set \[ \Delta_{\varepsilon} = \{(\delta_1, \ldots, \delta_s) : \delta_1 + \cdots + \delta_s \leq \varepsilon\} \,,\] if
$(p_1, \ldots, p_s, z)$ is a coherent assessment on $\mathcal{S} \cup \{\mathcal{C}(\mathcal{S})\}$ such that
\[
p_1 \geq 1 - \delta_1 \,,\, p_2 \geq 1 - \delta_2 \,,\, \ldots \,,\, p_s \geq 1 - \delta_s \,,
 \]
 it follows $p_1 + \cdots + p_s  \geq s - \varepsilon$, so that $(p_1, \ldots, p_s) \in \mathcal{L}_{1-\varepsilon}$, and hence $z = P[\mathcal{C}(\mathcal{S})] \geq 1-\varepsilon$. Therefore $\mathcal{S} \; \Rightarrow_p \; \mathcal{C}(\mathcal{S})$. \\
(ii) We observe that, denoting by $\Pi_s$ the set of coherent assessments on $\mathcal{S}$, in the case of logical independence it holds that $\Pi_s = [0,1]^s$, while in case of some logical dependencies among  the events $E_i,H_i, i=1,\ldots,s$ we have $\Pi_s \subseteq [0,1]^s$. Then $\mathcal{L}_{1-\varepsilon} = \{(p_1, \ldots, p_s) \in \Pi_s : p_1 + \cdots +
p_s  \geq s - \varepsilon\}$, with
$\mathcal{L}_{1-\varepsilon} \subseteq \{(p_1, \ldots, p_s) \in [0,1]^s : p_1 + \cdots +
p_s  \geq s - \varepsilon\}$, and with $\mathcal{L}_{1-\varepsilon} \neq \emptyset$ by p-consistency of $\mathcal{F}_s$.  Then, by the same reasoning as in case (i), we still obtain $\mathcal{S} \; \Rightarrow_p \; \mathcal{C}(\mathcal{S})$. \\
(b) By Remark \ref{ENTAIL-SUB}, for each nonempty subfamily $\mathcal{S}$ of $\mathcal{F}_n$, we have that $\mathcal{F}_n$ p-entails $\mathcal{S}$; then, as $\mathcal{S}$ p-entails $\mathcal{C}(\mathcal{S})$, by applying Definition \ref{PE} with $B|A=C(\mathcal{S})$,  it follows $\mathcal{F}_n \;\Rightarrow_p \; \mathcal{C}(\mathcal{S})$.
\end{proof}
 The next result characterizes in the setting of coherence Adams' notion of \linebreak p-entailment of a conditional event $E|H$ from a family $\mathcal{F}_n$; moreover, it provides a probabilistic semantics to the notion of entailment given in Definition \ref{ENT-DP} for conditional objects. We observe that the equivalence of assertions 1 and 5 were already given (without proof) in \cite[Thm 7]{GiSa10}.
\begin{theorem}\label{ENTAIL-CS}{\rm
Let be given a p-consistent family $\mathcal{F}_n = \{E_1|H_1, \ldots, E_n|H_n\}$ and a conditional event $E|H$. The following assertions are equivalent: \\
1. $\mathcal{F}_n$ p-entails $E|H$; \\
 2. The assessment $\mathcal{P}=(1,\ldots,1,z)$ on $\mathcal{F}=\mathcal{F}_n \cup \{E|H\}$, where $P(E_i|H_i)=1,$ $i=1,\ldots,n, P(E|H)=z$, is coherent if and only if $z=1$; \\
 3. The assessment $\mathcal{P}=(1,\ldots,1,0)$ on $\mathcal{F}=\mathcal{F}_n \cup \{E|H\}$, where $P(E_i|H_i)=1,$ $i=1,\ldots,n, P(E|H)=0$, is not coherent; \\
 4. Either there exists a nonempty $\mathcal{S} \subseteq \mathcal{F}_n$  such that $\mathcal{C}(\mathcal{S})$ implies $E|H$, or $H \subseteq E$. \\
5.  There exists a nonempty $\mathcal{S} \subseteq \mathcal{F}_n$ such that $\mathcal{C}(\mathcal{S})$ p-entails $E|H$.
}\end{theorem}
\begin{proof} We will prove that $1. \Rightarrow 2. \Rightarrow 3. \Rightarrow 4. \Rightarrow 5. \Rightarrow 1.$ \\
$(1. \Rightarrow 2.)$ Assuming that $\mathcal{F}_n$ p-entails $E|H$, then $EH \neq \emptyset$, so that the assessment $z=1$ on $E|H$ is coherent; moreover, the assessment $(1,\ldots,1,z)$ on $\mathcal{F}_n \cup \{E|H\}$, where $z=P(E|H)$, is coherent if and only if $z=1$. In fact, if by absurd the assessment $(1,\ldots,1,z)$ were coherent for some $z<1$, then given any $\varepsilon$, such that $1-\varepsilon > z$, the condition
\[
P(E_i|H_i) = 1 \,,\; i=1,\ldots,n  \; \Longrightarrow \; P(E|H) > 1-\varepsilon \,,
\]
which is necessary for p-entailment of $E|H$ from $\mathcal{F}_n$, would be not satisfied.
$(2. \Rightarrow 3.)$ It immediately follows by the previous step, when $z=0$. \\
$(3. \Rightarrow 4.)$ As the assessment $\mathcal{P}=(1,\ldots,1,0)$ on $\mathcal{F}=\mathcal{F}_n \cup \{E|H\}$ is not coherent, by applying Algorithm 1 to the pair $(\mathcal{F},\mathcal{P})$, at a certain  iteration, say the $k$-th one, the initial system  $\Sigma_k$ will be not solvable and the algorithm will stop. The system $\Sigma_k$ will be associated with a pair, say $(\mathcal{U}_k,\mathcal{P}_k)$, where  $\mathcal{U}_k = \mathcal{S}_k \cup \{E|H\}$, with $\mathcal{S}_k \subseteq \mathcal{F}_n$, and where $\mathcal{P}_k = (1,\ldots,1,0)$ is the (incoherent) sub-vector of $\mathcal{P}$ associated with $\mathcal{U}_k$. We have two cases: (i) $\mathcal{S}_k \neq \emptyset$; (ii) $\mathcal{S}_k = \emptyset$. \\
(i) For the sake of simplicity, we set $\mathcal{S}_k=\{E_1|H_1, \ldots, E_s|H_s\}$, with $s \leq n$; then, we denote by $C_1, \ldots, C_m$
the constituents generated by the family $\mathcal{S}_k \cup \{E|H\}$ and contained in $H_1 \vee \cdots \vee H_s \vee H$. Now, we will prove that $\mathcal{C}(\mathcal{S}_k) \subseteq E|H$. \\
We have
$\; \mathcal{C}(\mathcal{S}_k) = (E_1H_1 \vee H_1^c)\wedge \cdots \wedge (E_sH_s \vee H_s^c) \,|\, (H_1 \vee \cdots \vee H_s)\;$
and, if it were $\mathcal{C}(\mathcal{S}_k) \nsubseteq E|H$, then equivalently, by applying formula (\ref{NOT-GN}) to the conditional events $\mathcal{C}(\mathcal{S}_k), E|H$, there would exist at least a constituent, say $C_1$, of the following kind: \\
(a) $C_1 = B_1A_1 \cdots B_rA_rA_{r+1}^c \cdots A_{s}^cE^cH$, $1 \leq r \leq s$, or  \\
(b) $C_1 = H_1^cH_2^c \cdots H_{s}^cE^cH$, or \\
(c) $C_1 = B_1A_1 \cdots B_rA_rA_{r+1}^c \cdots A_{s}^cH^c$, $1 \leq r \leq s$, \\
where $B_i|A_i = E_{j_i}|H_{j_i},\, i=1,\ldots,s$, for a suitable permutation $(j_1, \ldots, j_s)$ of $(1,\ldots,s)$. \\
For each one of the three cases, (a), (b), (c), the vector
$(\lambda_1, \lambda_2, \ldots, \lambda_m) = (1,0, \ldots, 0)$,
associated with the constituents $C_1, C_2, \ldots, C_m$, would be a solution of the system $\Sigma_k$; then, $\Sigma_k$ would be solvable, which would be a contradiction; hence, it cannot exist any constituent of kind (a), or (b), or (c); therefore, $\mathcal{C}(\mathcal{S}_k) \subseteq E|H$. Hence the assertion 4 is true for $\mathcal{S} = \mathcal{S}_k$. \\
(ii) First of all we observe that, concerning $E$ and $H$, if $EH = \emptyset$, then the unique coherent assessment on $E|H$ is $P(E|H)=0$; if $EH \neq \emptyset$ and $H \nsubseteq E$, then the assessment $P(E|H)=p$ on  $E|H$ is coherent for every $p \in [0,1]$; if  $H \subseteq E$, then the unique coherent assessment on $E|H$ is $P(E|H)=1$. Now, as $\mathcal{S}_k = \emptyset$, the algorithm stops with $\mathcal{U}_k = \{E|H\}$; then, the assessment $P(E|H)=0$ is incoherent, which amounts to $H \subseteq E$. \\
$(4. \Rightarrow 5.)$  If $\mathcal{C}(\mathcal{S}) \subseteq E|H$ for some nonempty $\mathcal{S} \subseteq \mathcal{F}_n$, then, observing that by p-consistency of $\mathcal{F}_n$ the assessment $P[\mathcal{C}(\mathcal{S})] = 1$ is coherent, $\mathcal{C}(\mathcal{S})$ p-entails $E|H$. Otherwise, if $H \subseteq E$, then the unique coherent assessment on $E|H$ is $P(E|H)=1$ and trivially $\mathcal{C}(\mathcal{S})$ p-entails $E|H$ for every nonempty $\mathcal{S} \subseteq \mathcal{F}_n$. \\
$(5. \Rightarrow 1.)$ Assuming that $\mathcal{C}(\mathcal{S})$ p-entails $E|H$ for some nonempty $\mathcal{S} \subseteq \mathcal{F}_n$, by recalling Theorem \ref{ENT-QC}, we have $\mathcal{F}_n \;\Rightarrow_p \; \mathcal{C}(\mathcal{S}) \;\Rightarrow_p \; E|H$. Therefore, by the transitive property of p-entailment, we have $\mathcal{F}_n \;\Rightarrow_p \; E|H$.
\end{proof}

\begin{remark} {\rm Notice that, given two conditional events $A|B, E|H$, with $AB \neq \emptyset$ (so that the family $\{A|B\}$ is p-consistent), by applying Theorem \ref{ENTAIL-CS} with $n=1,\, \mathcal{F}_1 = \{A|B\}$, it follows
\[
A|B \; \Rightarrow_p \; E|H \;\Longleftrightarrow\; A|B \,\subseteq \, E|H \;\; \mbox{or} \;\; H \subseteq E \,.
\]
}\end{remark}
We observe that p-consistency of $\mathcal{F}_n\cup\{E|H\}$ is not sufficient for the \linebreak p-entailment of $E|H$ from $\mathcal{F}_n$. More precisely, we have
\begin{theorem}{\rm
\label{ALTERN}
Given a p-consistent family of $n$ conditional events $\mathcal{F}_n = \{E_1|H_1,$ $ \ldots, E_n|H_n\}$ and a conditional event $E|H$, the following assertions are equivalent: \\
a) the family $\mathcal{F}_n\cup\{E|H\}$ is  p-consistent; \\
b) exactly one of the following alternatives holds: \\
$(i)$ $\mathcal{F}_n$ p-entails $E|H$ ;\\
$(ii)$ the assessment $\mathcal{P}=(1,\ldots,1,z)$ on $\mathcal{F}_n\cup \{E|H\}$, where $P(E_i|H_i)=1, i=1,\ldots, n, P(E|H)=z$, is  coherent for every $z\in [0,1]$. }
\end{theorem}
\begin{proof}
(a $\Rightarrow$ b) Assuming $\mathcal{F}_n\cup\{E|H\}$   p-consistent, if statement $(i)$ holds, by Theorem \ref{ENTAIL-CS} the assessment $\mathcal{P}_0=(1,\ldots,1,0)$ on $\mathcal{F}=\mathcal{F}_n \cup \{E|H\}$ is not coherent; hence, statement $(ii)$ does not hold. If $(i)$ doesn't hold, by Theorem \ref{ENTAIL-CS} the assessment $\mathcal{P}_0=(1,\ldots,1,0)$ on $\mathcal{F}_n \cup \{E|H\}$ is coherent. Moreover, by p-consistency of $\mathcal{F}_n\cup\{E|H\}$ the assessment $\mathcal{P}_1=(1,\ldots,1,1)$ on $\mathcal{F}_n \cup \{E|H\}$ is coherent. Hence, by the extension Theorem of conditional probabilities (see also \cite{BiGi05}, Theorem 7)  the assessment $\mathcal{P}=(1,\ldots,1,z)$ on $\mathcal{F}_n \cup \{E|H\}$ is coherent for every $z\in [0,1]$; in other words, statement $(ii)$ holds. \\
(b $\Rightarrow$ a) If statement $(i)$ holds, then by Theorem \ref{ENTAIL-CS} the assessment $(1,\ldots,1,1)$ on $\mathcal{F}_n \cup \{E|H\}$ is coherent; hence $\mathcal{F}_n \cup \{E|H\}$ is p-consistent. If statement $(ii)$ holds, then again the assessment $(1,\ldots,1,1)$ on $\mathcal{F}_n \cup \{E|H\}$ is coherent; hence $\mathcal{F}_n \cup \{E|H\}$ is p-consistent.
\end{proof}
When $\mathcal{F}_n\cup \{E|H\}$ is not p-consistent, both statements $(i)$ and $(ii)$, in Theorem  \ref{ALTERN}, do not hold. We observe that, given any pair $(\mathcal{F}_n, E|H)$, if $\mathcal{F}_n$ is p-consistent, then  we have the following three possible cases:
\begin{enumerate}
\item[($c_1$)]  $\mathcal{F}_n\cup \{E|H\}$ is p-consistent and $\mathcal{F}_n$ p-entails $\{E|H\}$;
\item[$(c_2)$] $\mathcal{F}_n\cup \{E|H\}$ is p-consistent and $\mathcal{F}_n$ does not p-entail $\{E|H\}$;
\item[$(c_3)$] $\mathcal{F}_n\cup \{E|H\}$ is not p-consistent.
\end{enumerate}
These three cases are characterized in the next result.
\begin{theorem}{\rm
\label{ALTERN-BIS}
Given a  p-consistent family of $n$ conditional events $\mathcal{F}_n$ and a further conditional event $E|H$, let $\mathcal{P}=(1,\ldots,1,z)$ be a probability assessment on $\mathcal{F}_n\cup \{E|H\}$, where $P(E_i|H_i)=1, i=1,\ldots, n, P(E|H)=z$. Then, exactly one of the following statements is true:\\
$(a_1)$ $\mathcal{P}=(1,\ldots,1,z)$ is coherent if and only if $z = 1$;\\
$(a_2)$ $\mathcal{P}=(1,\ldots,1,z)$ is coherent for every $z \in [0,1]$; \\
$(a_3)$ $\mathcal{P}=(1,\ldots,1,z)$ is coherent if and only if $z = 0$. }
\end{theorem}
\begin{proof} We show that $(c_1)$  is equivalent to $(a_1)$, $(c_2)$  is equivalent to $(a_2)$ and $(c_3)$  is equivalent to  $(a_3)$. \\  The case $(c_1)$, by Theorem \ref{ENTAIL-CS}, amounts to say that the assessment $\mathcal{P}=(1,\ldots,1,z)$ is coherent if and only if $z = 1$, which is statement $(a_1)$. \\ The case $(c_2)$ amounts to say that the assessment $\mathcal{P}=(1,\ldots,1,z)$ is coherent for every $z \in [0,1]$, which is statement $(a_2)$.  Indeed, if $\mathcal{F}_n \cup \{E|H\}$ is p-consistent and $E|H$ is not p-entailed from $\mathcal{F}_n$, then by condition (ii) in Theorem \ref{ALTERN} the assessment $\mathcal{P}=(1,\ldots,1,z)$ is coherent for every $z \in [0,1]$. Conversely, if the assessment $\mathcal{P}=(1,\ldots,1,z)$ is coherent for every $z \in [0,1]$, then the assessments $\mathcal{P}_1=(1,\ldots,1,1)$ and $\mathcal{P}_0=(1,\ldots,1,0)$ on $\mathcal{F}_n \cup \{E|H\}$ are both coherent and hence $\mathcal{F}_n \cup \{E|H\}$ is p-consistent and $E|H$ is not p-entailed from $\mathcal{F}_n$. \\ The case $(c_3)$ amounts to say that the assessment $(1,\ldots,1,1)$ on $\mathcal{F}_n \cup \{E|H\}$ is not coherent; that is, the assessment $(1,\ldots,1,0)$ on $\mathcal{F}_n \cup \{E^c|H\}$ is not coherent; that is, by Theorem \ref{ENTAIL-CS}, $\mathcal{F}_n$ p-entails $E^c|H$, or equivalently, the assessment $(1,\ldots,1,p)$ on $\mathcal{F}_n \cup \{E^c|H\}$ is coherent if and only if $p=1$, which amounts to say that the assessment $(1,\ldots,1,z)$ on $\mathcal{F}_n \cup \{E|H\}$ is coherent if and only if $z=0$, which is  statement $(a_3)$.
\end{proof}
We observe that, if the assessment $(1,\ldots,1,z)$ on $\mathcal{F}_n \cup \{E|H\}$ is coherent for some $z \in (0,1)$, then by the previous theorem the assessment $(1,\ldots,1,z)$ is coherent for every $z \in [0,1]$.
\section{The class $\mathcal{K}$}
In this section we deepen the analysis of the quasi conjunction and the Goodman-Nguyen inclusion relation. Given a family  $\mathcal{F}_n$ and a further conditional event $E|H$, we set
\begin{equation}\label{CLASS-K}
\mathcal{K}(\mathcal{F}_n, E|H) = \{\mathcal{S} \subseteq \mathcal{F}_n, \mathcal{S} \neq \emptyset : \mathcal{C}(\mathcal{S}) \subseteq E|H\} \,.
\end{equation}
As the family $\mathcal{F}_n$ is finite, the class $\mathcal{K}(\mathcal{F}_n, E|H)$ is finite too. For the sake of simplicity, we simply denote $\mathcal{K}(\mathcal{F}_n, E|H)$ by $\mathcal{K}$. In what follows we will study the class $\mathcal{K}$, by giving some results which prove the properties listed below. \\
\textbf{Properties of class} $\mathcal{K}$
\begin{enumerate}
\item Given a nonempty subset $\mathcal{S}$ of $\mathcal{F}_n$ and a probability assessment $\mathcal{P}=(1,\ldots,1,0)$ on $\mathcal{S} \cup \{E|H\}$, where $P(E_i|H_i)=1$ for each $E_i|H_i \in \mathcal{S}$ and $P(E|H)=0$, we have: $\mathcal{S} \in \mathcal{K} \; \Longleftrightarrow \; \mathcal{P} \notin \I$, where $\I$ is the convex hull associated with the pair $(\mathcal{S} \cup \{E|H\},\mathcal{P})$;
\item it may happen that the class $\mathcal{K}$ is empty and $\F_n \Rightarrow_p E|H$;
\item the class $\mathcal{K}$ is \emph{additive}:
$\mathcal{S'}\in \mathcal{K}, \;\mathcal{S''}\in \mathcal{K} \Longrightarrow \mathcal{S'} \cup \mathcal{S''} \in \mathcal{K}$;
\item given any $\mathcal{S} \in \mathcal{K},\, \mathcal{U} \notin \mathcal{K}$, if $\mathcal{S} \subset \mathcal{U}$, then $\mathcal{U} \setminus \mathcal{S}\notin \mathcal{K}$;
\item given any nonempty subsets $\mathcal{S}$ and $\Gamma$ of $\mathcal{F}_n$, if $\mathcal{C}(\mathcal{S}) \subseteq \mathcal{C}(\Gamma)$ and $\mathcal{S}\cup \Gamma\in \mathcal{K}$, then $\mathcal{S} \in \mathcal{K}$;
    \item if $H \nsubseteq E$ and $\mathcal{F}_n \; \Rightarrow_p \; E|H$, then the class $\mathcal{K}$ is nonempty and has a greatest element $\S^*$;
        \item for every subset $\mathcal{S} \in \mathcal{K}$, we have: $\mathcal{S} \; \Rightarrow_p \; E|H$.
\end{enumerate}
Property 1 follows by showing that, for any given nonempty subset $\mathcal{S}$ of $\mathcal{F}_n$,  the relation $\mathcal{C}(\mathcal{S}) \; \subseteq \; E|H$, that is $\mathcal{S} \in \mathcal{K}$, is equivalent to the condition $\mathcal{P} \notin \mathcal{I}$
as proved in the following result.

\begin{theorem}\label{QC-SYSTEM}{\rm
Given a p-consistent family of $s$ conditional events $\mathcal{S} = \{E_1|H_1,$ $ \ldots,E_s|H_s\}$, with $s \geq 1$, and   a further conditional event $E|H$, let $\mathcal{P}=(1,\ldots,1,0)$ be a probability assessment on $\mathcal{F}=\mathcal{S} \cup \{E|H\}$. Moreover, let $\mathcal{I}$ be the convex hull of the points $Q_h$ associated with the pair $(\mathcal{F},\mathcal{P})$ and let $\Sigma$ be the starting system when applying Algorithm 1. We have
\begin{equation}\label{INCL-SOLV}
\mathcal{P} \notin \mathcal{I} \,,\;\; i.e. \,,\; \Sigma \;\; \mbox{unsolvable} \; \Longleftrightarrow \; \mathcal{C}(\mathcal{S}) \; \subseteq \; E|H \,. \end{equation}
}\end{theorem}
\begin{proof}
$(\Rightarrow)$ If $\Sigma$ is unsolvable, then $\mathcal{P}=(1,\ldots,1,0)$ is not coherent and, by applying the part $(3 \Rightarrow 4)$ of Theorem \ref{ENTAIL-CS} with $\mathcal{F}_n=\mathcal{S}$, Algorithm 1 stops at $\Sigma_k = \Sigma,\, \mathcal{S}_k=\mathcal{S}$. Then, we have $\mathcal{C}(\mathcal{S}_k)= \mathcal{C}(\mathcal{S}) \subseteq E|H$. \\
$(\Leftarrow)$  If $\Sigma$ is solvable, then the point $\mathcal{P}$ belongs to the convex hull $\mathcal{I}$; that is, $\mathcal{P}$ is a linear convex combination of the points $Q_h$. Then, as $\mathcal{P}$ is a vertex of the unitary hypercube $[0,1]^{s+1}$, which contains $\mathcal{I}$, the condition $\mathcal{P} \in \mathcal{I}$ is satisfied if and only if there exists a point $Q_h$, say $Q_1$, which coincides with $\mathcal{P}$. Then, there exists at least a constituent $C_1$ of the kind (a), or (b), or (c), as defined in the proof of Theorem \ref{ENTAIL-CS}, and this implies that $\mathcal{C}(\mathcal{S}) \nsubseteq E|H$.
\end{proof}

We give below an example to illustrate the previous result.
\begin{example}[\emph{Cut Rule}]
Given three logically independent events $A, B, C$, let us consider the family
$\mathcal{S} = \{C|AB, B|A\}$ and the further conditional event $C|A$. Of course, $\mathcal{S}$ p-entails $C|A$.
Let $\mathcal{P}=(1,1,0)$ be a probability assessment on $\mathcal{F}=\mathcal{S} \cup \{C|A\}$.
The constituents contained in $\H_3=A$ are
\[
C_1=ABC \,,\; C_2=ABC^c \,,\; C_3=AB^cC \,,\; C_4=AB^cC^c \,,;
\]
The associated points $Q_h$'s are
\[
Q_1=(1,1,1) \,,\; Q_2=(0,1,0) \,,\; Q_3=(1,0,1) \,,\; Q_4=(1,0,0) \,.
\]
\begin{figure}[!ht]
\label{fig:CUT}
\begin{center}
\includegraphics[scale=\myscale]{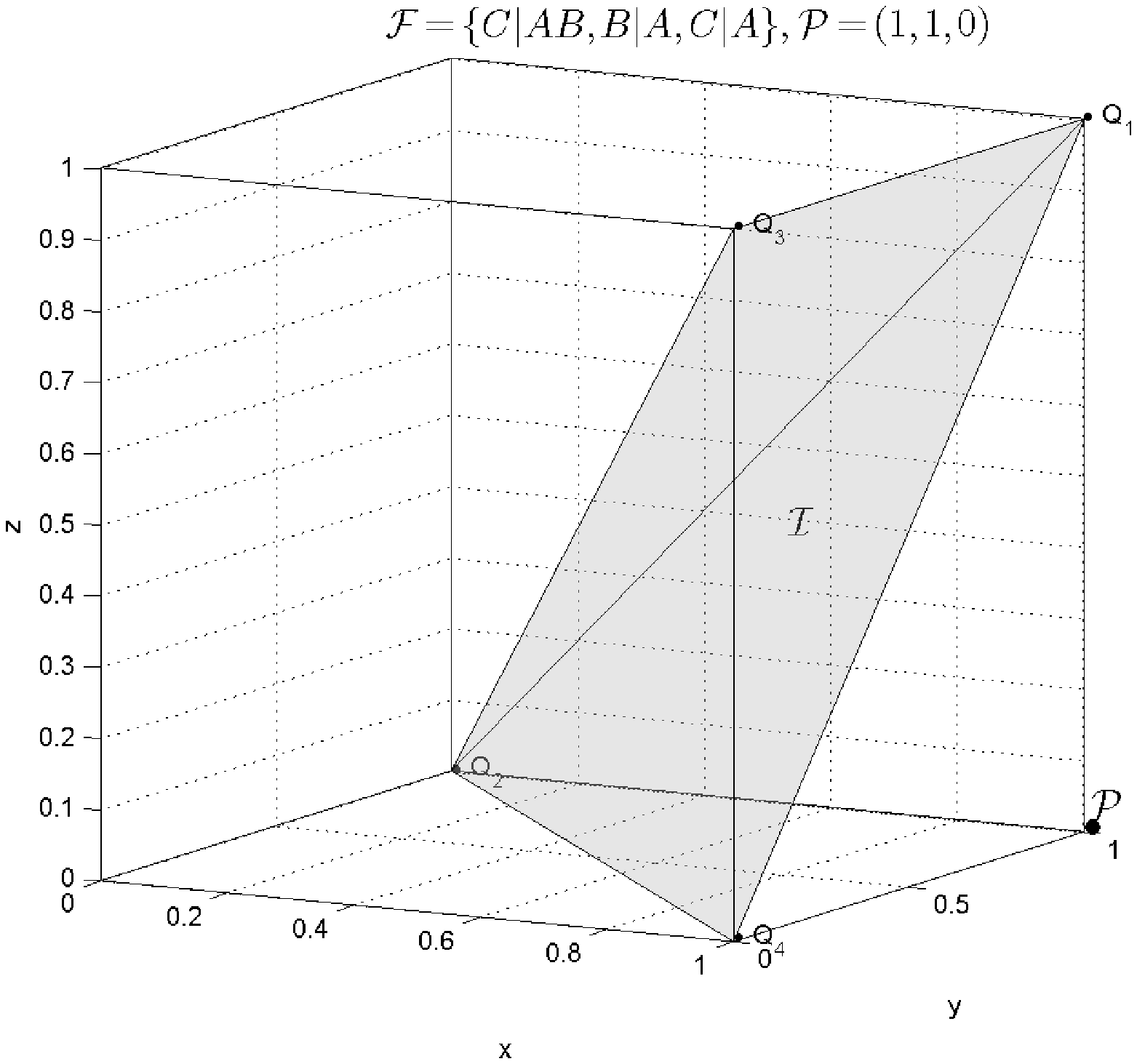}
\caption{The convex hull $\I$  associated with the pair $(\F,\P)$.}
\label{fig1}
\end{center}
\end{figure}
In Figure \ref{fig:CUT}  the  convex hull $\mathcal{I}$  of the points $Q_h$ associated with the pair $(\mathcal{F},\mathcal{P})$ is shown.  As $\C(C|AB,B|A)=BC|A\subseteq C|A$,   the system $\Sigma$ is not solvable; that is, as shown in Figure \ref{fig:CUT}, we have $\mathcal{P}=(1,1,0)\notin \mathcal{I}$.
\end{example}
The property 2 is proved by the following
\begin{proposition}
Given any pair $(\F_n, E|H)$, the following conditions are compatible: (i) the class $\mathcal{K}$ is empty; (ii) $\F_n \Rightarrow_p E|H$.
\end{proposition}
\begin{proof}
We observe that, if $H \subseteq E$, then the family $\mathcal{F}_n$ trivially p-entails $E|H$; at the same time it may happen that there doesn't exist any (nonempty) subset $\mathcal{S}$ of $\mathcal{F}_n$ such that $\mathcal{C}(\mathcal{S}) \subseteq E|H$; for instance, if $\mathcal{F}_n = \{B|A\}$, with $B|A \nsubseteq E|H$ and $H \subseteq E$, then $\{B|A\}$ trivially p-entails $E|H$ and the class $\mathcal{K}$ is empty.
\end{proof}
The property 3 is proved in the next result.
\begin{theorem}\label{CS_UNION}{\rm
Given two nonempty subsets $\mathcal{S}'$ and $\mathcal{S}''$ of $\mathcal{F}_n$ and a conditional event $E|H$, assume that $\mathcal{S}' \in \mathcal{K},\, \mathcal{S}'' \in \mathcal{K}$. Then, $\mathcal{S}' \cup \mathcal{S}'' \in \mathcal{K}$.
}\end{theorem}
\begin{proof} By the associative property, it is $\mathcal{C}(\mathcal{S}' \cup \mathcal{S}'')=\mathcal{C}(\mathcal{C}(\mathcal{S}')\,, \mathcal{C}(\mathcal{S}''))$; then: \\ (i)  $\mathcal{C}(\mathcal{S}' \cup \mathcal{S}'')$ {\em true} implies $\mathcal{C}(\mathcal{S}')$ {\em true}, or $\mathcal{C}(\mathcal{S}'')$ {\em true}; hence, $E|H$ is {\em true}; \\ (ii) $E|H$ {\em false} implies that $\mathcal{C}(\mathcal{S}')$ and $\mathcal{C}(\mathcal{S}'')$ are both {\em false}; hence, $\mathcal{C}(\mathcal{S}' \cup \mathcal{S}'')$ is {\em false}. Therefore: $\mathcal{C}(\mathcal{S}' \cup \mathcal{S}'') \subseteq E|H$. Hence $\mathcal{S}' \cup \mathcal{S}'' \in \mathcal{K}$; that is $\mathcal{K}$ is additive.
\end{proof}
The property 4 is proved in the following
\begin{corollary}{\rm
Given two subsets $\mathcal{S}$ and $\mathcal{U}$ of $\mathcal{F}_n$, assume that $\mathcal{S} \subset \mathcal{U}$, with $\mathcal{C}(\mathcal{S}) \subseteq E|H,\, \mathcal{C}(\mathcal{U}) \nsubseteq E|H$. Then: $\mathcal{C}(\mathcal{U} \setminus \mathcal{S}) \nsubseteq E|H$.
}\end{corollary}
\begin{proof} The proof immediately follows by Theorem \ref{CS_UNION} by observing that, if \linebreak $\mathcal{C}(\mathcal{U} \setminus \mathcal{S}) \subseteq E|H$, then  $\mathcal{C}[\mathcal{S} \, \cup \, (\mathcal{U} \setminus \mathcal{S})] = \mathcal{C}(\mathcal{U}) \subseteq E|H$, which contradicts the hypothesis.
\end{proof}
In order to prove the property 5, we recall that $A|H \subseteq B|K$ amounts to $H^cB^cK = AHB^cK = AHK^c = \emptyset$  (see Remark \ref{NOT-GN}); thus \[
\mathcal{C}(A|H,B|K)=(AH \vee H^cBK) \,|\, (H \vee K) \,.
\]
Moreover, as shown by Table \ref{TABLE-QC-INCL}, it holds that
\begin{equation}\label{QC-GN}
A|H \subseteq \mathcal{C}(A|H,B|K) \subseteq B|K \,.
\end{equation}
\begin{table}[!ht]
\center
\[
\begin{array}{cc}
\begin{array}{l|ccccc}
\mbox{Constituent}      &  A|H  &   \mathcal{C}(A|H,B|K)   & B|K\\
\hline
H^cK^c   &       \Void  & \Void     & \Void     \\
AHBK    &        \True  & \True     & \True       \\
H^cBK   &        \Void  & \True     & \True        \\
A^cHBK  &        \False & \False    & \True        \\
A^cHK^c  &       \False  & \False    & \Void       \\
A^cHB^cK   &    \False   & \False    & \False       \end{array} 
\end{array}
\]
\caption{Truth-Table of   $A|H,  \mathcal{C}(A|H,B|K), B|K$ in case $A|H\subseteq B|K$. }
\label{TABLE-QC-INCL}
\end{table}

Then, we have
\begin{theorem}\label{ENT-DIMIN}{\rm
Let $\mathcal{F}_n$ be a family of $n$ conditional events, with $n \geq 2$, and $E|H$ be a further conditional event. Moreover, let $\mathcal{S}$ and $\Gamma$ be two nonempty subfamilies of $\mathcal{F}_n$ such that $\mathcal{C}(\mathcal{S}) \subseteq \mathcal{C}(\Gamma)$ and $\mathcal{C}(\mathcal{S} \cup \Gamma) \subseteq E|H$. Then, we have $\mathcal{C}(\mathcal{S}) \subseteq E|H$.
}
\end{theorem}
\begin{proof} By the associative property of quasi conjunction we have $\mathcal{C}(\mathcal{C}(\mathcal{S}), \mathcal{C}(\Gamma)) = \mathcal{C}(\mathcal{S} \cup \Gamma)$; then, by applying (\ref{QC-GN}), with $A|H = \mathcal{C}(\mathcal{S})$ and $B|K = \mathcal{C}(\Gamma)$, we obtain $\mathcal{C}(\mathcal{S}) \subseteq \mathcal{C}(\mathcal{S} \cup \Gamma) \subseteq \mathcal{C}(\Gamma)$. As $\mathcal{C}(\mathcal{S} \cup \Gamma) \subseteq E|H$, it follows $\mathcal{C}(\mathcal{S}) \subseteq E|H$.
\end{proof}
The property 6 is proved in the next result.
\begin{theorem}\label{THM-GUE}{\rm
Given a family of $n$ conditional events $\mathcal{F}_n$ and a further conditional event $E|H$, with $H \nsubseteq E$, assume that $\mathcal{F}_n$ p-entails $E|H$. Then, the class $\mathcal{K}$ is nonempty and has a greatest element $\mathcal{S}^*$.
}\end{theorem}
\begin{proof}
Since  $\mathcal{F}_n$ p-entails $E|H$, by assertion 4 in Theorem \ref{ENTAIL-CS}, $\mathcal{K}$ is nonempty; moreover, by Theorem \ref{CS_UNION}, $\mathcal{K}$ is additive. Then, denoting by $\mathcal{S}^*$ the union of all elements of $\mathcal{K}$, it holds that $\mathcal{S}^* \in \mathcal{K}$. Of course, $\mathcal{S}^*$ is the greatest element of $\mathcal{K}$; that is, $\mathcal{S} \subseteq \mathcal{S}^*$, for every $\mathcal{S} \in \mathcal{K}$.
\end{proof}
The property 7 is proved in the next result.
\begin{proposition}
Given a family of $n$ conditional events $\mathcal{F}_n$ and a further conditional event $E|H$, assume that the class $\mathcal{K}$ is nonempty. Then, for every subset $\mathcal{S} \in \mathcal{K}$, we have: $\mathcal{S} \; \Rightarrow_p \; E|H$.
\end{proposition}
\begin{proof}The condition $\mathcal{S} \in \mathcal{K}$ amounts to $\mathcal{C}(\mathcal{S}) \subseteq E|H$; then, by the step $(4. \Rightarrow 5.)$ of Theorem \ref{ENTAIL-CS}, it follows that: $\mathcal{S} \in \mathcal{K}$ implies $\mathcal{S} \Rightarrow_p E|H$.
\end{proof}
\begin{remark}{\rm Assuming $H \nsubseteq E$, by Theorem \ref{ENTAIL-CS},  $\mathcal{F}_n$ p-entails $E|H$ if and only if there exists a nonempty subset $\mathcal{S}_k$ of  $\mathcal{F}_n$ such that, when applying Algorithm~1 to the assessment $\mathcal{P}=(1,\ldots,1,0)$ on $\mathcal{F}=\mathcal{F}_n \cup \{E|H\}$, the system $\Sigma_k$ associated with the family $\mathcal{S}_k \cup \{E|H\}$ is not solvable and the algorithm will stop. In the next result we show that $\mathcal{S}_k$ coincides with the greatest element of $\mathcal{K}$, $\mathcal{S^*}$.
}\end{remark}
\begin{theorem}\label{ALG-GREAT}{\rm
Given a family of $n$ conditional events $\mathcal{F}_n$ and a further conditional event $E|H$, with $H \nsubseteq E$, assume that $\mathcal{F}_n$ p-entails $E|H$. Moreover, let $\mathcal{P}=(1,\ldots,1,0)$ be a probability assessment on $\mathcal{F}=\mathcal{F}_n \cup \{E|H\}$. Then, $\mathcal{P}$ is not coherent and by applying Algorithm 1 to the pair $(\mathcal{F},\mathcal{P})$, the nonempty subset $\mathcal{S}_k$, associated with the iteration where Algorithm 1 stops, coincides with the greatest element $\mathcal{S}^*$ of $\mathcal{K}$.
}
\end{theorem}
\begin{proof}
By assertion 3 of Theorem \ref{ENTAIL-CS}, $\mathcal{P}$ is not coherent; moreover, by the step $(3. \Rightarrow 4.)$ of Theorem \ref{ENTAIL-CS}, we have $\mathcal{C}(\mathcal{S}_k) \subseteq E|H$, so that $\mathcal{S}_k \in \mathcal{K}$ and hence $\mathcal{S}_k \subseteq \mathcal{S}^*$. In order to prove that $\mathcal{S}_k = \mathcal{S}^*$, we will show that $\mathcal{S}_k \subset \mathcal{S}^*$ gives a contradiction.
If $\mathcal{S}_k=\mathcal{F}_n$, then $\mathcal{S}_k=\mathcal{S^*}$. Assume that $\mathcal{S}_k \subset \mathcal{F}_n$ and, by absurd, that $\mathcal{S}_k \subset \mathcal{S}^*$. By applying Algorithm 1 to the pair $(\mathcal{F},\mathcal{P})$ we obtain a partition   $\Gamma^{(1)},\Gamma^{(2)},\ldots, \Gamma^{(k)}$,with $k>1$, such that
\[
\mathcal{F}_n\cup \{E|H\}=\Gamma^{(1)}\cup \Gamma^{(2)}\cup\cdots\cup \Gamma^{(k)};\;\;\;\Gamma^{(i)}\cap \Gamma^{(j)}=\emptyset,\, \mbox{ if } i\neq j \,,
\]
where $\Gamma^{(k)} = \mathcal{U}_k = \mathcal{S}_k \cup \{E|H\}$. Then $\mathcal{S}^* \cap \Gamma^{(k)} = \mathcal{S}_k$. Now, by the (absurd) hypothesis $\mathcal{S}_k \subset \mathcal{S}^*$ it would follow $\mathcal{S}^* \cap \Gamma^{(j)} \neq \emptyset$ for at least one index $j < k$. Denoting by $r$ the minimum index such that $\mathcal{S}^* \cap \Gamma^{(r)} \neq \emptyset$ and defining
\[
\mathcal{F}^{(r)} = \Gamma^{(r)} \cup \cdots \cup \Gamma^{(k)} \,,\;\; \mathcal{F}^{(r+1)} = \Gamma^{(r+1)} \cup \cdots \cup \Gamma^{(k)} \,, \]
we would have $\mathcal{S}^* \subseteq \mathcal{F}^{(r)} \,,\,\; \mathcal{S}^* \setminus \mathcal{F}^{(r+1)} \neq \emptyset$; moreover, the system $\Sigma^{(r)}$ associated with the pair $(\mathcal{F}^{(r)},\mathcal{P}^{(r)})$ would be solvable. Defining
\[
J = \{j: E_j|H_j \in \mathcal{S}^*\} \,,\;\; \mathcal{F}_J = \mathcal{S}^* \cup \{E|H\} \,,
\] and denoting by $\mathcal{P}_J$ the sub-vector of $\mathcal{P}$ associated with $\mathcal{F}_J$, by Theorem \ref{GILIO-93} it would follow that the system $\Sigma_J$ associated with the pair $(\mathcal{F}_J,\mathcal{P}_J)$ would be solvable and by Theorem \ref{QC-SYSTEM} we would have $\mathcal{C}(\mathcal{S}^*) \nsubseteq E|H$, which is absurd because  $\mathcal{S}^* \in \mathcal{K}$. Therefore $\mathcal{S}_k = \mathcal{S}^*$.
\end{proof}
Based on Theorem \ref{THM-S-PCONSIST} and  Theorem \ref{ALG-GREAT}, we give below a suitably modified version of Algorithm 1, which allows to examine the following aspects: \\
$(i)$ checking for p-consistency of $\mathcal{F}_n$; \\ $(ii)$ checking for
 p-entailment of $E|H$ from $\mathcal{F}_n$; \\ $(iii)$ computation of the greatest element  $\mathcal{S}^*$.
\begin{algorithm}\label{ALG-PENTAILS}
{\rm Let  be given the pair $(\mathcal{F}_n,E|H)$, with $\mathcal{F}_n=\{E_1|H_1,\ldots, E_n|H_n\}$ and $H\nsubseteq E$.
\begin{enumerate}
\item Set $\mathcal{P}_n=(1,1,\ldots,1)$, where  $P(E_i|H_i)=1$, $i=1,\ldots, n$. Check the coherence of $\mathcal{P}_n$ on $\mathcal{F}_n$ by Algorithm 1.
      If  $\mathcal{P}_n$ on $\mathcal{F}_n$ is coherent then $\mathcal{F}_n$ is  p-consistent, set  $\mathcal{F}=\mathcal{F}_n \cup \{E|H\}$,  $\mathcal{P}=(\mathcal{P}_n,0)$ and go to step 2; otherwise $\mathcal{F}_n$ is not p-consistent and  procedure stops.
\item \label{i} Construct the system $\Sigma$ associated with $(\mathcal{F},\mathcal{P})$ and check its solvability.
\item If the system $\Sigma$ is not solvable then $\mathcal{F}_n$ p-entails $E|H$, $\mathcal{S^*}=\mathcal{F} \setminus \{E|H\}$
 and the procedure stops; otherwise (that is, $\Sigma$ solvable) compute the set $I_0$ defined in formula (\ref{PHI-I0}).
\item \label{ii} If $I_0 = \emptyset$ then $\mathcal{F}_n$ does not p-entail $E|H$  and the procedure
stops; otherwise set  $(\mathcal{F}, \mathcal{P}) = (\mathcal{F}_0, \mathcal{P}_0 )$ and go to step \ref{i}.
\end{enumerate}
}\end{algorithm}
\begin{remark}
We point out that in Algorithm \ref{ALG-PENTAILS} the family $\mathcal{F}_n \cup \{E|H\}$ must be intended as the family of $n+1$ conditional events $\{E_1|H_1,\ldots,E_n|H_n,E|H\}$ even if $E_i|H_i=E|H$ for some $i$; hence, at step 3, where $\mathcal{F}=\mathcal{S} \cup \{E|H\}$ for some $\mathcal{S} \subseteq \mathcal{F}_n$, to set  $\mathcal{S^*}=\mathcal{F} \setminus \{E|H\}$ must be intended as to set $\mathcal{S^*}=\mathcal{S}$.
\end{remark}
As for similar algorithms existing in literature, which analyze the problem of checking coherence and  propagation, also with Algorithm \ref{ALG-PENTAILS} the checking of p-consistency and of p-entailment is intractable when the family $\mathcal{F}_n$ is large. A detailed analysis of the different levels of complexity in this kind of problems has been given in \cite{BGLS05}. As a further aspect, Algorithm~\ref{ALG-PENTAILS} provides the greatest element (if any) $\mathcal{S}^*$ of the class $\mathcal{K}$. We observe that,  if in step 1  $\mathcal{F}_n$ results p-consistent, then the set $\mathcal{S}^*$ (if any) is determined in at most $n$ cycles of the algorithm.
The example below illustrates Algorithm~\ref{ALG-PENTAILS}.
\begin{example}
Given four logically independent events $A, B, C, D$, let us consider the family $\mathcal{F}_5 = \{C|B, B|A, A|(A \vee B), B|(A \vee B),D|A^c\}$ and the further conditional event $C|A$. By applying Algorithm~\ref{ALG-PENTAILS} to the pair $(\mathcal{F}_5, C|A)$, it can be proved that the assessment $\mathcal{P}_5=(1,1,1,1,1)$ on $\mathcal{F}_5$ is coherent; hence $\mathcal{F}_5$ is p-consistent. Moreover, as
\[\mathcal{C}(\mathcal{F}_5) = (ABC\vee A^cB^cD)|\Omega  \nsubseteq C|A \,,
\]
the starting system $\Sigma$, associated with the pair $(\mathcal{F}, \mathcal{P})$, where $\mathcal{F} = \mathcal{F}_5 \cup \{C|A\}$ and $\mathcal{P} = (1,1,1,1,1,0)$, is solvable and we have $I_0 = \{1,2,3,4\}$. The procedure goes to step \ref{i}, with $(\mathcal{F}, \mathcal{P})$=$(\mathcal{F}_0, \mathcal{P}_0)$, where
\[
\mathcal{F}_0 = \{C|B, B|A, A|(A \vee B), B|(A \vee B)\} \cup \{C|A\} \,,\;\; \mathcal{P}_0 = (1,1,1,1,0) \,.
\]
As
\[
\mathcal{C}(\{C|B, B|A, A|(A \vee B), B|(A \vee B)\}) = ABC|(A \vee B) \subset C|A \,,
\]
the system $\Sigma$ associated with the pair $(\mathcal{F}_0, \mathcal{P}_0)$ is not solvable; then, by Theorem \ref{ENTAIL-CS}, $\mathcal{F}_5$ p-entails $C|A$ and the procedure stops, with
\[
\mathcal{S}^*=\mathcal{F}_5\setminus \{D|A^c\}=\{C|B, B|A, A|(A \vee B), B|(A \vee B)\} \,.
\]
Moreover, by setting $\mathcal{S}^* = \mathcal{S}_1$ and defining
\[
\mathcal{S}_2 = \mathcal{S}_1 \setminus \{B|(A \vee B)\} = \{C|B, B|A, A|(A \vee B)\} \,,
 \]
it holds: $\mathcal{C}(\mathcal{S}_2) = ABC|(A \vee B) \; \subset \; B|(A \vee B)$; hence, by Theorem \ref{ENT-DIMIN},  $\mathcal{C}(\mathcal{S}_2) \subset C|A$, that is $\mathcal{S}_2\in \mathcal{K}$. We also observe that,  defining \[
\mathcal{S}_3 = \mathcal{S}_1 \setminus \{B|A \} = \{C|B, A|(A \vee B),B|(A\vee B)\} \,,
 \]
it is: $\mathcal{C}(\mathcal{S}_3) = ABC|(A \vee B) \; \subset \; B|A$; hence, by Theorem \ref{ENT-DIMIN},  $\mathcal{C}(\mathcal{S}_3) \subset C|A$, that is $\mathcal{S}_3\in \mathcal{K}$.
Finally, it could be proved that, for every nonempty subset $\mathcal{S}$ of $\mathcal{F}_5$, with $\mathcal{S} \notin \{\mathcal{S}_1,\mathcal{S}_2,\mathcal{S}_3\}$, it holds that $\mathcal{C}(\mathcal{S}) \nsubseteq C|A$, i.e. $\mathcal{S} \notin \mathcal{K}$; therefore $\mathcal{K}=\{\mathcal{S}_1,\mathcal{S}_2,\mathcal{S}_3\}$ and $\mathcal{S}_i$ p-entails $C|A$, $i=1,2,3$.
\end{example}
We observe that the problem of determining the class $\mathcal{K}$ by refining the methodology applied in the previous example seems intractable because the cardinality of $\mathcal{K}$ may be $2^n - 1$, as shown by the example below.
\begin{example}
Given the pair $(\mathcal{F}_n,E|H)$, assume that $E_i|H_i \subseteq E|H$ for every $i=1,\ldots,n$. Then, it holds that: (i) $\{E_i|H_i\} \in \mathcal{K}$ for every $i=1,\ldots,n$; \\ (ii) $ \{E_i|H_i,E_j|H_j\} \in \mathcal{K}$ for every $\{i,j\} \subseteq \{1,\ldots,n\}$; and so on. In this case the cardinality of $\mathcal{K}$ is $2^n - 1$.
\end{example}
\section{Conclusions}
In this paper we have studied the probabilistic entailment in the setting of coherence. In this framework we have analyzed the role of quasi conjunction and the Goodman-Nguyen inclusion relation for conditional events. By deepening some results given in a previous paper we have shown that, given any finite family $\mathcal{F}$ of conditional events and any nonempty subset $\mathcal{S}$ of $\mathcal{F}$, the quasi conjunction $\mathcal{C}(\mathcal{S})$ is p-entailed by $\mathcal{F}$. We have also examined the probabilistic semantics of QAND rule. Then, we have characterized p-entailment by many equivalent assertions. In particular, given any conditional event $E|H$, with $H \nsubseteq E$, we have shown the equivalence between p-entailment of $E|H$ from $\mathcal{F}$ and the existence of a nonempty subset $\mathcal{S}$ of $\mathcal{F}$ such that $\mathcal{C}(\mathcal{S})$ p-entails $E|H$. For a pair $(\mathcal{F},E|H)$ we have examined some alternative theorems related with p-consistency and p-entailment. Moreover, we have introduced  the (possibly empty) class $\K$ of the subsets $\mathcal{S}$ of $\mathcal{F}$ such that $\mathcal{C}(\mathcal{S})$ implies $E|H$. We have shown that the class $\K$ satisfies many properties, in particular, every $\mathcal{S} \in \mathcal{K}$ p-entails $E|H$, $\K$ is additive and has a greatest element which can be determined by applying Algorithm \ref{ALG-PENTAILS}. Finally, we have illustrated the theoretical results and Algorithm \ref{ALG-PENTAILS}, by examining an example. Interestingly, some of the results concerning the class $\K$ are connected with results on a similar but different class introduced in \cite[Sec 5]{2002BGLS-JANCL}; further work should compare these classes and clarify the differences between them. \\ \ \\
{\bf Acknowledgments} \\
The authors thank the anonymous referees for they valuable criticisms and suggestions.

\bibliographystyle{model1b-num-names}

\bibliography{biblioijar}
\end{document}